\documentclass[11pt]{amsart}
\usepackage{mathrsfs}
\usepackage{amsmath,amstext,amsfonts,verbatim}
\usepackage{cite}
\usepackage{amssymb}

\usepackage[latin1]{inputenc}
\usepackage{amsmath,amsthm,amssymb}

\newtheorem{theorem}{Theorem}[section]
\newtheorem{lemma}[theorem]{Lemma}
\newtheorem{proposition}[theorem]{Proposition}
\newtheorem{corollary}[theorem]{Corollary}
\theoremstyle{definition}

\theoremstyle{remark}

\numberwithin{equation}{section}

\title[Integral Operators, Fock-Sobolev Spaces, Multipliers on Gauss-Sobolev Spaces]{Integral Operators on Fock-Sobolev Spaces via Multipliers on Gauss-Sobolev Spaces}
\author[B. D. Wick and S. Wu]{Brett D. Wick and Shengkun Wu}

\address{Brett D. Wick, Department of Mathematics and Statistics, Washington University in St. Louis, MO 63130, USA}
\email{ wick@math.wustl.edu}
\address{Shengkun Wu, College of Mathematics and Statistics, Chongqing University, Chongqing, 401331, PR China}
\email{shengkunwu@foxmail.com}

 \thanks{Brett D. Wick's research is supported in part by a National Science Foundation DMS grants \# 1560955 and \# 1800057 and Australian Research Council -- DP 190100970. Shengkun Wu's research is supported by CSC201906050022.}
 \thanks{}
 \thanks{}

\date{}

\begin{document}

\begin{abstract} In this paper, we obtain an isometry between the Fock-Sobolev space and the Gauss-Sobolev space. As an application, we use multipliers on the Gauss-Sobolev space to characterize the boundedness of an integral operator on the Fock-Sobolev space.

\end{abstract} \maketitle

\section{Introduction }
Let $\mathbb{C}^{n}$ be the complex $n$ dimensional space and $d v$ be the ordinary volume measure on $\mathbb{C}^{n}$.
If $z=\left(z_{1}, \ldots, z_{n}\right)$ and $w=\left(w_{1}, \ldots, w_{n}\right)$ are points in $\mathbb{C}^{n},$ we write
$$
z \cdot \overline{w}=\sum_{j=1}^{n} z_{j} \overline{w}_{j}, \quad|z|=(z \cdot \overline{z})^{1 / 2}.
$$
Let Gaussian measure be
$$
d \lambda(z)=\pi^{-n} e^{-|z|^{2}} d v(z).
$$
Denote by $L^{2}(\mathbb{C}^{n},d\lambda)$ the set of square integrable functions with respect to $d\lambda$.
The Fock space $F^2:=F^{2}\left(\mathbb{C}^{n}\right)$ consists of all entire functions $f$ on the complex Euclidean space $\mathbb{C}^{n}$
such that
$$
\|f\|_{F^{2}}=\left(\int_{\mathbb{C}^{n}}|f(z)|^{2} d \lambda(z)\right)^{\frac{1}{2}}<\infty.
$$
$F^{2}$ is a closed subspace of the Hilbert space $L^{2}(\mathbb{C}^{n},d\lambda)$ with inner product
$$\langle f, g\rangle_{F^2}=\int_{\mathbb{C}^{n}} f(z) \overline{g(z)}  d \lambda(z).$$
The orthogonal projection $P : L^{2}(\mathbb{C}^{n},d\lambda) \to F^{2}$ is given by
$$
P f(z)=\frac{1}{\pi^{n}} \int_{\mathbb{C}^{n}} f(w) K(z, w) e^{-|w|^{2}} d v(w),
$$
where $K(z, w)=e^{z \cdot \overline{w}}$ is the reproducing kernel of $F^{2}$.

In what follows we use standard multi-index notation. For an $n$-tuple
$\alpha=\left(\alpha_{1}, \ldots, \alpha_{n}\right)$ of non-negative integers, we write
$$
|\alpha|=\alpha_{1}+\cdots+\alpha_{n}, \quad \alpha !=\alpha_{1} ! \cdots \alpha_{n} !.
$$
If $z=\left(z_{1}, \cdots, z_{n}\right),$ then $z^{\alpha}=z_{1}^{\alpha_{1}} \cdots z_{n}^{\alpha_{n}}$ and $\partial^{\alpha}=\partial_1^{\alpha_1}\cdots\partial_{n}^{\alpha_{n}}$, where $\partial_{j}$ denotes the partial differentiation with respect to the $j$-th component.

For any positive integer $m$ we consider the space $F^{2,m}$, called the Fock-Sobolev space, consisting of entire functions $f$ on $\mathbb{C}^{n}$ such that
$$
\|f\|_{F^{2, m}} :=\sum_{|\alpha| \leq m}\left\|\partial^{\alpha} f\right\|_{F^{2}}<\infty,
$$
where $\|\cdot \|_{F^2}$ is the norm in $F^{2}$.

Creation and annihilation operators on the Fock space are important operators in quantum field theory. However, these two operators are unbounded operators on the Fock space. For the study of unbounded operators, it is important to understand the domain of definition of the operator. In fact, creation and annihilation operators are bounded from the Fock-Sobolev space with order $1$ to the Fock space. So, the study of Fock-Sobolev spaces can help us to study these two operators.

We know that the Bargmann transform is an isometry between $L^{2}(\mathbb{R}^n)$ and $F^{2}(\mathbb{C}^n)$. The Bargmann transform connects Weyl psedo-differential operators on $L^{2}(\mathbb{R}^n)$ to Toeplitz operators on $F^{2}$.  For example, in \cite{coburn1994}, the authors  studied the boundedness of Toeplitz operators on the Fock space. In that paper, the authors used the Berezin model in the Fock space to obtain the lower bound for the Toeplitz operators. On the other hand, they used Weyl psedo-differential operators to obtain the upper bound for the Toeplitz operators. The Bargmann transform connects these two spaces together and allows tools from one side to be transported to the other for analysis.

Let $W^{2,m}(dx)$ be the classical Sobolev space on $\mathbb{R}^n$. A natural question arises: Is the Bargmann transform an isomorphism between $F^{2,m}$ and $W^{2,m}(\mathbb{R}^n)$? We will answer this question in Section 2. We will recall some facts about the Gauss-Sobolev space in the Gaussian Harmonic Analysis. The study of Gaussian Harmonic Analysis arise from probability theory, quantum mechanics, and differential geometry. A key operator in the theory of Gaussian Harmonic Analysis is the Ornstein-Uhlenbeck operator. Giving the Ornstein-Uhlenbeck operator, we can define Gaussian Bessel potential which is important for our proof, see \cite{Urbina} and \cite{LP}.
In Section 2, we will obtain an isometry between the Fock-Sobolev space and Gauss-Sobolev space. Because of the isometry between the Fock-Sobolev space and Gauss-Sobolev space, we will connect questions in these two spaces together.

For $\varphi \in F^{2}$ consider the integral operator
$$
S_{\varphi} f(z)=\int_{\mathbb{C}^{n}} f(w) e^{z \cdot \overline{w}} \varphi(z-\overline{w}) d \lambda(w),
$$
for any $f\in F^{2,m}$.
In \cite{zhu2015}, Zhu used the Bargman transform to transfer some singular integral operators to $S_{\varphi}$ and proposed an open question about the boundedness of $S_{\varphi}$. In \cite{wick2019}, the authors gave a necessary and sufficient condition for $S_\varphi$ to be bounded on $F^2$. In this paper, we consider the same problem in Fock-Sobolev spaces.

In Section 3, we will study the multipliers on the Gauss-Sobolev space. Then, in Section 4, we will obtain an isomorphism between multipliers in the Gauss-Sobolev space and the set of bounded $S_{\varphi}$. Then we use the conclusion in the Gauss-Sobolev space to characterize the boundedness of the integral operator on the Fock-Sobolev space and study other properties.

The multipliers on the Sobolev spaces has been studied in \cite{Mazya}. In \cite{Liu}, the authors studied the Gaussian Capacity theory in Gauss-Sobolev space with order 1. In this paper, we will use the idea in \cite{Mazya} and some operators in the Gaussian Harmonic analysis to obtain the boundedness of multiplication operator between two Gauss-Sobolev spaces.

\section{Gauss-Sobolev spaces}
In this section, we introduce the Gauss-Bargmann transform and show that the Gauss-Bargmann transform is an isometry that maps the Gauss-Sobolev space to the Fock-Sobolev space. On the other hand, we show that the Bargmann transform is not an isomorphism between the Fock-Sobolev space and the Sobolev space.

First, we introduce the Gauss Sobolev space. Let Gaussian measure $d\gamma$ on $\mathbb{R}^{n}$ be given by
$$
d\gamma(x)=\frac{1}{{(2\pi)}^{\frac{n}{2}}}e^{-\frac{|x|^2}{2}}dx.
$$
For any positive integer $m$, the Gauss-Sobolev space $W^{2,m}(\gamma)$ is the completion of $C_{0}^{\infty}(\mathbb{R}^n)$ with respect to the norm
$$
\|f\|_{W^{2,m}(\gamma)}=\sum_{0\leq|\alpha|\leq m}\left[\int_{\mathbb{R}^n}|\partial^{\alpha}f(x)|^2d\gamma(x)  \right]^{\frac{1}{2}}.
$$
In \cite{Bogachev}, some properties of the Gauss-Sobolev space are discussed.

For any multi-index $\beta=\left(\beta_{1}, \dots, \beta_{n}\right)$, the Hermite function is defined to be
$$ H_{\beta}(x)=\prod_{i=1}^{n}(-1)^{\beta_{i}} e^{x_{i}^{2}} \frac{\partial^{\beta_{i}}}{\partial x_{i}^{\beta_{i}}}\left(e^{-x_{i}^{2}}\right).$$
Then the normalized Hermite function with respect to the Gaussian measure is given by:
$$
h_{\beta}(x)=\frac{1}{(2^{|\beta|} \beta !)^{1 / 2}} H_{\beta}\left(\frac{x}{\sqrt{2}}\right).
$$
That is to say $$\int_{\mathbb{R}^n} h_{\beta}(x)h_{\alpha}(x)d\gamma(x)=\delta_{\alpha \beta},$$
where $\delta_{\alpha \beta}=1$ if $\alpha=\beta$ and $\delta_{\alpha \beta}=0$ if $\alpha\neq\beta$.

For any multi-index $\alpha$ one easily computes that
$$
\partial^{\alpha} h_{\beta}(x)=\left\{\begin{array}{ll}{\left(\prod\limits_{j=1}^{n} \beta_{j}\left(\beta_{j}-1\right) \cdots\left(\beta_{j}-\alpha_{j}+1\right)\right)^{1 / 2} h_{\beta-\alpha}(x),} & {\text {if } \alpha_{j} \leq \beta_{j}, \forall j=1, \ldots, n} \\ {0,} & {\text{otherwise.}}\end{array}\right.
$$
By \cite[Proposition 1.5.4]{Bogachev}, we know that the linear space generated by Hermite polynomials is dense in $W^{2,m}(\gamma)$.

For $z\in \mathbb{C}$, let $e_{\beta}(z)=\frac{z^{\beta}}{\sqrt{\beta!}}$ be the basis of the Fock space; we know that
$$
\partial^{\alpha} e_{\beta}(z)=\left\{\begin{array}{ll}{\left( \prod\limits_{j=1}^{n}\beta_{j}\left(\beta_{j}-1\right) \cdots\left(\beta_{j}-\alpha_{j}+1\right)\right)^{1 / 2} e_{\beta-\alpha}(z),} & {\text {if } \alpha_{j} \leq \beta_{j}, \forall j=1, \ldots, n} \\ {0,} & {\text{otherwise.}}\end{array}\right.
$$
From these two observations, we know that
\begin{equation}\label{2.1}
\|e_{\beta}\|_{F^{2,m}}=\|h_{\beta}\|_{W^{2,m}(\gamma)},
\end{equation}
for any $\beta$.
We define the Gauss-Bargmann transform $G$ mapping the linear span of $\{h_{\beta}\}$ to $F^{2,m}$  such that
$$
Gh_{\beta}=e_{\beta}.
$$
\begin{theorem}
The Gauss-Bargmann transform $G$ is an isometry from the Gauss-Sobolev space $W^{2,m}(\gamma)$ to the Fock-Sobolev space $F^{2,m}$.
\end{theorem}
\begin{proof}
We know that $\{e_{\beta}\}$ and $\{h_{\beta}\}$ are complete orthogonal sets in $F^{2,m}$ and $W^{2,m}(\gamma)$ respectively. The statement then follows form \eqref{2.1}.
\end{proof}

We want to contrast this new transform with the more well-known Bargmann transform.  Recall that the Bargmann transform is an isometry from $L^2(\mathbb{R}^n,dx)$ to $F^2$ such that
$$B f(z)=\left(\frac{2}{\pi}\right)^{\frac{n}{4}} \int_{\mathbb{R}^{n}} f(x) e^{2 x \cdot z-x^{2}-\frac{z^{2}}{2}} dx,$$
where $z^2=z_1^2+z_2^2+\dots +z_n^2$, $x^2=x_1^2+x_2^2+\dots +x_n^2$ and $x \cdot z=x_1 z_1+x_2 z_2+\dots+x_n z_n$.
Let
$$\widetilde{h}_{\beta}=\left(\frac{2}{\pi}\right)^{\frac{n}{4}}\frac{1}{\sqrt{2^{\beta}\beta!}}e^{-|x|^2}H_{\beta}(\sqrt{2}x),$$
we know that $B\widetilde{h}_{\beta}=e_{\beta}$, see \cite[Theorem 6.8]{zhu}.
That is to say
\begin{align*}
e_{\beta}&=B\widetilde{h}_{\beta}(z)\\
&=\left(\frac{2}{\pi}\right)^{\frac{n}{4}} \int_{\mathbb{R}^{n}} \left(\frac{2}{\pi}\right)^{\frac{n}{4}}\frac{1}{\sqrt{2^{\beta}\beta!}}e^{-|x|^2}H_{\beta}(\sqrt{2}x) e^{2 x \cdot z-x^{2}-\frac{z^{2}}{2}} dx\\
&=\left(\frac{2}{\pi}\right)^{\frac{n}{2}} \int_{\mathbb{R}^{n}} \frac{1}{\sqrt{2^{\beta}\beta!}}e^{\frac{-|x|^2}{4}}H_{\beta}\left(\frac{x}{\sqrt{2}}\right)
e^{ x \cdot z-\frac{x^{2}}{4}-\frac{z^{2}}{2}} \frac{1}{2^n}dx\\
&=\int_{\mathbb{R}^{n}} \frac{1}{\sqrt{2^{\beta}\beta!}}H_{\beta}\left(\frac{x}{\sqrt{2}}\right) e^{ x \cdot z-\frac{z^{2}}{2}} d\gamma(x)\\
&=\int_{\mathbb{R}^{n}} h_{\beta}(x) e^{ x \cdot z-\frac{z^{2}}{2}} d\gamma(x).\\
\end{align*}
By the argument above, we know that for any $f\in W^{2,m}(\gamma)$, we have
$$
Gf(z)=\int_{\mathbb{R}^{n}}f(x) e^{ x \cdot z-\frac{z^{2}}{2}} d\gamma(x).
$$
Similarly, for any $g\in F^{2,m}$, we have
$$
G^{-1}g(x)=\int_{\mathbb{C}^{n}}g(z) e^{ x \cdot \overline{z}-\frac{\overline{z}^{2}}{2}} d\lambda(z).
$$
Next, we will discuss the relationship between the Gauss-Bargmann transform and the Bargmann
transform; the key point will be that the order of smoothness matters for these operators.

Let $C_{\frac{1}{2}}$ be the composition operator form $L^2(\mathbb{R}^n, dx)$ to $L^2(\mathbb{R}^n, dx)$ such that
$C_{\frac{1}{2}}f(x)=f(\frac{x}{2})$, for any $f\in L^2(\mathbb{R}^n, dx)$.
Let $M_{\left(\frac{\pi}{2}\right)^ {\frac{n}{4}}\exp\left(\frac{|x|^2}{4}\right)}$ be the multiplication operator from $L^2(\mathbb{R}^n, dx)$ to $L^2(\mathbb{R}^n, d\gamma)$ such that
$$
M_{\left(\frac{\pi}{2}\right)^ {\frac{n}{4}}\exp\left(\frac{|x|^2}{4}\right)}f(x)=\left(\frac{\pi}{2}\right)^ {\frac{n}{4}}\exp\left(\frac{|x|^2}{4}\right)f(x).
$$
For simplicity of notation, we denote $M_{\left(\frac{\pi}{2}\right)^ {\frac{n}{4}}\exp\left(\frac{|x|^2}{4}\right)}$ with $M$.
\begin{proposition}\label{proposition 2.2}
The relationship between the Bargmann transform and the Gauss-Bargmann transform is given by
$$B=GMC_{\frac{1}{2}}.$$
\end{proposition}
\begin{proof}  This is simply a computation from the definitions of the operators involved.  For any $f\in L^{2}(\mathbb{R}^n,dx),$ we have
\begin{align*}
GMC_{\frac{1}{2}}f(z)&=\int_{\mathbb{R}^{n}}\left(\frac{\pi}{2}\right)^ {\frac{n}{4}}\exp\left(\frac{|x|^2}{4}\right)f\left(\frac{x}{2}\right) e^{ x \cdot z-\frac{z^{2}}{2}} d\gamma(x)\\
&=\int_{\mathbb{R}^{n}}\left(\frac{\pi}{2}\right)^ {\frac{n}{4}}\exp\left(\frac{|x|^2}{4}\right)f\left(\frac{x}{2}\right) e^{ x \cdot z-\frac{z^{2}}{2}}\frac{1}{{(2\pi)}^{\frac{n}{2}}}e^{-\frac{|x|^2}{2}}dx\\
&=Bf(z)
\end{align*}
to complete the proof.
\end{proof}

 To discuss the relationship between Sobolev spaces, Gauss-Sobolev spaces and Fock-Sobolev spaces, we need some basic facts about Fock-Sobolev spaces.
 The following theorem is a special case of \cite[Theorem 11]{zhu2012}.
\begin{theorem}\label{theorem 2.3}
Suppose $m$ is a non-negative integer, and $f$ is an entire
function on $\mathbb{C}^{n}$ . Then $f \in F^{2, m}$ if and only if every function $z^{\alpha} f(z)$ is in $F^{2}$,
where $|\alpha|=m$ . Moreover, there is a positive constant $c$ such that
$$
c^{-1}\big\||z|^{m} f\big\|_{F^2} \leq \|f\|_{F^{2,m}} \leq c\big\||z|^{m} f\big\|_{F^2}
$$
for all $f \in F^{2, m}$.
\end{theorem}
 Let $A_j$ and $A_{j}^*$ be two unbounded operators on $F^2$ such that $A_jf(z)=\partial_{z_j}f(z)$ and $A_{j}^*f(z)=z_jf(z)$. By \cite[Lemma 6.13]{zhu}, we have
\begin{equation}\label{2.2}
B\partial_{x_j}B^{-1}=A_j-A_j^* \text{ and } BM_{x_j}B^{-1}=\frac{1}{2}(A_j+A_j^*).
\end{equation}
For any $f\in F^{2,m}$, by Theorem \ref{theorem 2.3}, we have
$$
\|A^*_jf\|_{F^{2,m-1}}= \|z_jf\|_{F^{2,m-1}}\lesssim \||z|^{m-1}z_jf\|_{F^2}\lesssim \|f\|_{F^{2,m}}.
$$
We obtain that $A^*_j$ is bounded form $F^{2,m}$ to $F^{2,m-1}$. That $A_j$ is bounded from $F^{2,m}$ to $F^{2,m-1}$ follows from the definition of Fock-Sobolev spaces.

We also need a theorem about Sobolev spaces.
We define the $(p, m)$-capacity of a compact set $K \subset \mathbb{R}^{n}$ by
$$
C_{p, m}(K)=\inf \left\{\|f\|_{L^{p}(\mathbb{R}^n)}^{p}: \ f \in L^{p}(\mathbb{R}^n),\ f \geq 0,\ B_{m} f \geq 1 \text { on } K\right\},
$$
where $B_{m}$ is the Bessel potential of order $m$.
By \cite[pg. 16]{Mazya}, we have
\begin{equation}\label{2.3}
C_{p, m}(K) \approx \inf \left\{\|u\|_{W^{p,m}(dx)}^{p}:\ u \in C_{0}^{\infty}(\mathbb{R}^n),\ u \geq 1 \text { on } K\right\}.
\end{equation}
Recall that $C_{0}^{\infty}(\mathbb{R}^n)$ is the set of smooth functions on $\mathbb{R}^n$ with compact support.
\begin{theorem}[\!\!\protect{\cite[Theorem 1.2.2]{Mazya}}]
\label{theorem 2.4}
Let $p \in(1, \infty), m\in\mathbb{N}$ and let $\mu$ be a measure in $\mathbb{R}^{n}$.
Then the best constant $C$ in
$$\int_{\mathbb{R}^n}|u(x)|^{p} d \mu(x) \leq C\|u\|_{W^{p,m}(dx)}^{p}, \quad u \in C_{0}^{\infty}(\mathbb{R}^n),$$
is equivalent to
$$\sup _{K} \frac{\mu(K)}{C_{p, m}(K)},$$
where $K$ is an arbitrary compact set in $\mathbb{R}^{n}$.
\end{theorem}

\begin{proposition}The inverse of the Bargmann transform is bounded from the Fock-Sobolev space $F^{2,m}$ to the Sobolev space $W^{2,m}(dx)$. However, if $m\geq1$, the image of $B$ on $W^{2,m}(dx)$ is not contained in $F^{2,m}$.
\end{proposition}
\begin{proof}Suppose $f\in F^{2,m}$, we have $B^{-1}f=C_{\frac{1}{2}}^{-1}M^{-1}G^{-1}f.$ We only need to prove that $M^{-1}G^{-1}f\in W^{2,m}(dx).$
For any $\alpha=(\alpha_1,\dots,\alpha_n)$ with $|\alpha|\leq m$, there is a set of constants $\{c_\beta:\beta=(\beta_1,\beta_2,\dots,\beta_n)\}$ such that
\begin{align*}
\left\|\partial^{\alpha}M^{-1}G^{-1}f\right\|_{L^{2}(\mathbb{R}^n,dx)}
&=\left \|\sum_{\beta\leq\alpha}c_{\beta}x^{\beta} M^{-1}\partial^{\alpha-\beta}(G^{-1}f)\right\|_{L^{2}(\mathbb{R}^n,dx)}\\
&\lesssim\sum_{\beta\leq\alpha}\left\|x^{\beta} M^{-1}\partial^{\alpha-\beta}(G^{-1}f)\right\|_{L^{2}(\mathbb{R}^n,dx)}\\
&\lesssim\sum_{\beta\leq\alpha}\left\|x^{\beta}\partial^{\alpha-\beta}(G^{-1}f)\right\|_{L^{2}(\mathbb{R}^n,d\gamma)}\\
&\lesssim\sum_{\beta\leq\alpha}\left\|Gx^{\beta}G^{-1}G\partial^{\alpha-\beta}(G^{-1}f)\right\|_{F^{2}}.
\end{align*}
By direct computation, we know that
$$
M_{x^{\beta}}=2^{\beta}MC_{\frac{1}{2}}M_{x^{\beta}}C^{-1}_{\frac{1}{2}}M \text{ and } \partial^{\alpha-\beta}=\frac{1}{2^{\alpha-\beta}}MC_{\frac{1}{2}}\partial^{\alpha-\beta}C^{-1}_{\frac{1}{2}}M.
$$
Then
$$
\|\partial^{\alpha}M^{-1}G^{-1}f\|_{L^{2}(\mathbb{R}^n,dx)}
\lesssim\sum_{\beta\leq\alpha}\|Bx^{\beta}B^{-1}B\partial^{\alpha-\beta}B^{-1}f\|_{F^{2}}.
$$
By \eqref{2.2}, we have
$\|\partial^{\alpha}M^{-1}G^{-1}f\|_{L^{2}(\mathbb{R}^n,dx)}\lesssim \|f \|_{F^{2,|\alpha|}},$ which means that
$$
\|B^{-1}f\|_{W^{2,m}(dx)}\lesssim\|f\|_{F^{2,m}}.
$$

Next, we prove the second part of this theorem by contradiction.
Suppose $Bg\in F^{2,m}$ for any $g\in W^{2,m}(dx)$, that is to say $GMC_{\frac{1}{2}}g\in F^{2,m}(\gamma)$. Then, for any $g\in W^{2,m}(dx)$, we have
$Mg\in W^{2,m}(\gamma).$ Since $m\geq1$, we have
$\|\partial_{x_1}Mg\|_{L^{2}(\mathbb{R}^n,d\gamma)}<\infty.$
Since
$$\|\partial_{x_1}Mg\|_{L^{2}(\mathbb{R}^n,d\gamma)}=\|M\partial_{x_1}g+\frac{x_1}{2}Mg\|_{L^{2}(\mathbb{R}^n,d\gamma)}$$
and
$\|M\partial_{x_1}g\|_{L^{2}(\mathbb{R}^n,d\gamma)}=\|\partial_{x_1}g\|_{L^{2}(\mathbb{R}^n,dx)}\leq \|g\|_{W^{2,m}(dx)},$
we have
$$\|x_1 g\|_{L^{2}(\mathbb{R}^n,dx)}=\|x_1 Mg\|_{L^{2}(\mathbb{R}^n,d\gamma)}<\infty.$$
We have proved that $M_{x_1}g\in L^{2}(\mathbb{R}^n,dx)$ for any $g\in W^{2,m}(\gamma)$. Since $M_{x_1}$ is a closed operator, we know that $M_{x_1}$ is a bounded operator form $W^{2,m}(dx)$ to $L^{2}(\mathbb{R}^n,dx)$.

Let $d\mu=|x_1|^2dx$. For any positive $N$, let $K_{N}=\overline{B(0,N)}$, there is a $u_N \in C_{0}^{\infty}(\mathbb{R}^n)$ with $u_N = 1 \text { on } K_{N}$ and $u_N = 0 \text { on } B^{c}(0,N+1)$ such that
$$
\sup_{|\alpha|\leq m}\sup_{x}|\partial^{\alpha}u_{N}(x)|\leq c<\infty,
$$
where $c$ is independent of $N$. Thus we have
$$
\|u_N\|^2_{W^{2,m}(dx)}\lesssim|B(0,N+1)|\approx (N+1)^n.
$$
By \eqref{2.3}, we have
$$C_{2, m}(K_N)\lesssim(N+1)^n.$$
Then
$$\sup _{K} \frac{\mu(K)}{C_{2, m}(K)}\geq\frac{\mu(K_N)}{C_{2, m}(K_N)}\gtrsim\frac{\int_{K_N}|x_1|^2dx}{(N+1)^n}.$$
Since $[-\frac{N}{\sqrt{n}},\frac{N}{\sqrt{n}}]^n\subset K_{N}$, we have
$$\int_{K_N}|x_1|^2dx \geq \int_{[-\frac{N}{\sqrt{n}},\frac{N}{\sqrt{n}}]^n}|x_1|^2dx \gtrsim N^{n+2}.$$
That is to say $\sup _{K} \frac{\mu(K)}{C_{p, m}(K)}=\infty$, which is a contradiction by Theorem \ref{theorem 2.4}.
\end{proof}

\section{Multipliers on Gauss-Sobolev Spaces}
In this section, we study multipliers on Gauss-Sobolev spaces. First, we recall the definition of Gauss-Bessel potentials.

The Ornstein-Uhlenbeck differential operator is defined as
$$L=\sum_{j=1}^{n}\partial_{x_j}^2-\sum_{j=1}^{n}x_j\partial_{x_j} .$$
Some similar conclusions about the multipliers in classical Sobolev spaces have been proved in \cite{Mazya}. However, in the Gauss-Sobolev spaces, we need some properties of the Ornstein-Uhlenbeck differential operator.
Let $C_{n}$ be the closed subspace of $L^{2}\left(\gamma\right)$ generated by the linear combinations of $\left\{h_{\beta}:|\beta|=n\right\} .$
For any $s\geq0$, we consider the Gaussian-Bessel potentials defined by
\begin{equation}\label{4.1}
(I-L)^{-s / 2} f=\sum_{n=0}^{\infty}(1+n)^{-s / 2} J_{n} f, \quad \text { for } f \in L^2(\gamma),
\end{equation}
where $J_n$ is the orthogonal projection from $L^2(\gamma)$ to $C_{n}$.
The Gauss-Bessel potential space with order $s$ is
$$
L^{2,s}(\gamma)=\{f\in L^2(\gamma) : f=(I-L)^{-\frac{s}{2}}u\text{ for some } u\in L^{2}(\gamma)\}.
$$
The norm is defined as
$$
\|f\|_{L^{2,s}(\gamma)}=\|u\|_{L^2(\gamma)}, \text{ if } f=(I-L)^{-\frac{s}{2}}u.
$$
\begin{theorem}[\!\!\cite{LP}]\label{Theorem 3.1}
If $s$ is a non-negative integer, then
$$W^{2,s}(\gamma)=L^{2,s}(\gamma).$$
\end{theorem}
We also need a theorem of interpolation for Gauss-Sobolev spaces. Let $S=\{w\in \mathbb{C}: 0\leq \textnormal{Re}(w)\leq1\}$.
Given a compatible pair of Banach spaces $X_{0}$ and $X_{1},$ let $\mathcal{F}\left(X_{0}, X_{1}\right)$ be
the space of all functions $F$ from $\bar{S}$ into $X_{0}+X_{1}$ with the following properties:
\begin{enumerate}
\item $F$ is bounded and continuous on $\bar{S}$ and analytic in $S$;
\item $y \rightarrow F(k+i y)$ with $k=0,1$ are continuous from the real line into $X_{k}$.
\end{enumerate}
$\mathcal{F}\left(X_{0}, X_{1}\right)$ is clearly a vector space.  We provide $\mathcal{F}=\mathcal{F}\left(X_{0}, X_{1}\right)$ with the norm
$$\|F\|_{\mathcal{F}}=\max \left\{\sup _{y \in \mathbb{R}}\|F(iy)\|_{X_0}, \sup_{y\in\mathbb{R}}\|F(1+i y)\|_{X_1}\right\}.$$
Given $0 \leq \theta \leq 1,$ let $X_{\theta}$ be the space of vectors $v$ in $X_{0}+X_{1}$ such that
$v=f(\theta)$ for some $f$ in $\mathcal{F}\left(X_{0}, X_{1}\right) .$ We norm $X_{\theta}$ with
$\|v\|_{\theta}=\inf \left\{\|f\|_{\mathcal{F}}: v=f(\theta)\right\}.$
\begin{theorem}\label{Theorem 4.2}
Let $0\leq\theta\leq1$. Let $m_0\leq m_{\theta}\leq m_1$ be three non-negative constants with
$$
m_{\theta}=m_0(1-\theta)+m_1\theta,
$$
then
$$\left[L^{2,m_0}(\gamma),L^{2,m_1}(\gamma)\right]_{\theta}=L^{2,m_{\theta}}(\gamma),$$
where $\left[L^{2,m_0}(\gamma),L^{2,m_1}(\gamma)\right]_{\theta}$ is the interpolation space between $L^{2,m_0}(\gamma)$ and $L^{2,m_1}(\gamma)$.
\end{theorem}
\begin{proof} Since $L^{2,m_1}\subset L^{2,m_0}$, we know that $L^{2,m_1}+ L^{2,m_0}= L^{2,m_0}$. If $u\in L^{2,m_{\theta}}(\gamma)$, then there is $f\in L^2(\gamma)$ such that
$$u=(I-L)^{-m_{\theta} / 2}f.$$
For any $z\in \{w: 0\leq \textnormal{Re}(w)\leq1\}$, we define
$$
F(z)=\sum_{n=0}^{\infty}\left(\frac{1}{\sqrt{1+n}}\right)^{m_0(1-z)+m_1z} J_{n} f.
$$
It is easy to check that $F(z)$ is a vector-valued function from $\{w: 0\leq \textnormal{Re}(w)\leq1\}$ to
$L^{2,m_0}(\gamma)$ which is continuous on $\{w: 0\leq \textnormal{Re}(w)\leq1\}$ and analytic on $\{w: 0< \textnormal{Re}(w)<1\}$.
We know that
$$
F(\theta)=u.
$$
Then we have
$$\|u\|_{\theta}\leq \|F\|_\mathcal{F}\leq \|f\|_{L^{2}(\gamma)}=\|u\|_{L^{2,m_{\theta}}(\gamma)}.$$
Conversely, if $u\in \left[L^{2,m_0}(\gamma),L^{2,m_1}(\gamma)\right]_{\theta}$, then for any $\epsilon>0$,
there is a
$$F_{\epsilon}\in\mathcal{F}(L^{2,m_0}(\gamma), L^{2,m_1}(\gamma)).$$
with $F_{\epsilon}(\theta)=u$ such that
$$\|F_{\epsilon}\|_{\mathcal{F}}\leq\|u\|_{\theta}+\epsilon.$$
For any $g\in L^{2}(\gamma)$, $l\in \mathbb{N}$ and $z\in S$  , we define
$$H(z)=\sum_{n=0}^{l}(\sqrt{1+n})^{m_0(1-z)+m_1z}\langle F_{\epsilon}(z),J_{n}g \rangle_{L^{2}(\gamma)}.$$
It is easy to show that $H(z)$ is bounded and continuous on $\overline{S}$ and analytic in $S$. We consider
\begin{align*}
|H(ix)|=\left|\left\langle\sum_{n=0}^{l}(\sqrt{1+n})^{m_0(1-ix)+m_1ix} J_{n}F_{\epsilon}(ix),g \right\rangle_{L^{2}(\gamma)}\right|.
\end{align*}
Since $F_{\epsilon}(ix)\in L^{2,m_0}(\gamma)$ for any $x\in \mathbb{R}^n$. Then, for any $x\in \mathbb{R}$, there is $f_{x}\in L^{2}(\gamma)$  such that
$$F_{\epsilon}(ix)=(I-L)^{-\frac{m_0}{2}} f_x=\sum_{n=0}^{\infty}\left(\frac{1}{\sqrt{1+n}}\right)^{m_0} J_{n} f_x.$$
Then, we have
\begin{align*}
\sup_{x\in\mathbb{R}}\left\|\sum_{n=0}^{l}(\sqrt{1+n})^{m_0(1-ix)+m_1ix} J_{n}F_{\epsilon}(ix)\right\|_{L^2(\gamma)}
=&\sup_{x\in\mathbb{R}}\left\|\sum_{n=0}^{l}(\sqrt{1+n})^{-ixm_0+m_1ix} J_{n}f_x\right\|_{L^2(\gamma)}\\
\leq&\sup_{x\in\mathbb{R}}\|f_x\|_{L^2(\gamma)} =\sup_{x\in\mathbb{R}} \|F_{\epsilon}(ix)\|_{W^{2,m_0}(\gamma)}\\
\leq&\|F_{\epsilon}\|_{\mathcal{F}}.
\end{align*}
Then $\sup_{x\in \mathbb{R}}|H(ix)|\leq \|F_{\epsilon}\|_{\mathcal{F}} \|g\|_{L^2}$.
Similarly, we can obtain
$$\sup_{x\in \mathbb{R}}|H(1+ix)|\leq \|F_{\epsilon}\|_{\mathcal{F}} \|g\|_{L^2(\gamma)}.$$
By the Three Line Lemma, see \cite[pg. 28]{zhu1990}, we have $|H(\theta)|\leq \|F_{\epsilon}\|_{\mathcal{F}} \|g\|_{L^2(\gamma)}.$ That is to say
$$
\left|\left\langle \sum_{n=0}^{l}(\sqrt{1+n})^{m_{\theta}} J_{n}u,g \right\rangle_{L^{2}(\gamma)}\right|\leq \|F_{\epsilon}\|_{\mathcal{F}} \|g\|_{L^2(\gamma)},
$$
for any $l\in \mathbb{N}$ and $g\in L^2(\gamma)$. We obtain $\sum_{n=0}^{\infty}(\sqrt{1+n})^{m_{\theta}} J_{n}u\in L^2(\gamma)$ and
$$\left\|\sum_{n=0}^{\infty}(\sqrt{1+n})^{m_{\theta}} J_{n}u\right\|_{L^2(\gamma)}\leq \|F_{\epsilon}\|_{\mathcal{F}}.$$
Since $u=(I-L)^{-\frac{m_{\theta}}{2}} \left[\sum_{n=0}^{\infty}(\sqrt{1+n})^{m_{\theta}} J_{n}u\right]$, we have
$$\|u\|_{L^{2,m_{\theta}}}\leq \left\|\sum_{n=0}^{\infty}(\sqrt{1+n})^{m_{\theta}} J_{n}u\right\|_{L^2(\gamma)} \leq \|u\|_{\theta}+\epsilon$$
to complete the proof.
\end{proof}
Before proving the next lemma, we need some additional notation. For two multi-indexes $\alpha=(\alpha_1,\ldots,\alpha_n)$ and $\beta=(\beta_1,\ldots,\beta_n)$, if for all $k=1,\ldots,n$ we have $\alpha_k\leq\beta_k$, then we write
$$\alpha\leq\beta.$$

For any $u\in L^{1}_{loc}$, let $M_u$ denote the multiplication operator on $W^{2,m}(\gamma)$; $u$ is called a multiplier on $W^{2,m}(\gamma)$ if $M_u$ is bounded on $W^{2,m}(\gamma)$.
Let $M\left(W^{2,m}(\gamma) \rightarrow W^{2,m'}(\gamma)\right)$ denote the set of bounded multiplication operators from $W^{2,m}(\gamma)$ to
$W^{2,m'}(\gamma)$. If $m=m'$, we simply write $M\left(W^{2,m}(\gamma) \rightarrow W^{2,m}(\gamma)\right)$ as $MW^{2,m}(\gamma)$.
we have following simple lemma.
\begin{lemma}\label{lemma 3.3}
For any $u \in C^{\infty}(\mathbb{R}^n)$,  we have
$$
\|u\|_{MW^{2,m}(\gamma)}\lesssim \sum_{|\alpha|\leq m}\sup_{x}|\partial^{\alpha}u(x)|.
$$
\end{lemma}
\begin{proof}The proof is obvious as it follows from the definition of the norm of $W^{2,m}(\gamma)$, the product rule for differentiation and immediate estimates.
\end{proof}

\begin{lemma}\label{Lemma 4.3} Suppose that
$$
u \in MW^{2,m}(\gamma) \cap ML^{2}(\gamma).
$$
Then, for any multi-index $\alpha$ of order $|\alpha| \leq m$
$$
\partial^{\alpha} u \in M\left(W^{2,m}(\gamma) \rightarrow W^{2,m-|\alpha|}(\gamma)\right)
$$
and for any $\epsilon$, there is a $c(\epsilon)$ such that
$$\|\partial^{\alpha} u\|_{M(W^{2,m}(\gamma) \rightarrow W^{2,m-|\alpha|}(\gamma))}
\leq \varepsilon\|u\|_{ML^{2}(\gamma)}+c(\varepsilon)\|u\|_{MW^{2,m}(\gamma)}.$$
\end{lemma}
\begin{proof}If $\alpha=0$, the conclusion is obvious. We suppose that $\alpha\neq0$.
By \cite[pg. 39]{Mazya}, for any $g\in W^{2,m}(\gamma) $, just using the product rule applied to $ug$ and rearranging, we have
$$g \partial^{\alpha} u=\sum_{\{\beta: \alpha \geq \beta \geq 0\}} \frac{\alpha !}{\beta !(\alpha-\beta) !}\partial^{\beta}(u(-\partial)^{\alpha-\beta} g).$$
Then
\begin{align*}
\|g \partial^{\alpha} u\|_{W^{2,m-|\alpha|}(\gamma)}
&\lesssim \sum_{\{\beta: \alpha \geq \beta \geq 0\}}\|u \partial^{\alpha-\beta} g\|_{W^{2,m-|\alpha|+|\beta|}(\gamma)}\\
&\leq\sum_{\{\beta: \alpha \geq \beta \geq 0\}}\|u \|_{MW^{2,m-|\alpha|+|\beta|}(\gamma)}\|\partial^{\alpha-\beta} g\|_{W^{2,m-|\alpha|+|\beta|}(\gamma)}\\
&\leq\sum_{\{\beta: \alpha \geq \beta \geq 0\}}\|u \|_{MW^{2,m-|\alpha|+|\beta|}(\gamma)}\|g\|_{W^{2,m}(\gamma)}.\\
\end{align*}
Thus, by Theorem \ref{Theorem 4.2} and Theorem \ref{Theorem 3.1}, we have
\begin{equation}\label{4.2}
\begin{split}
&\|\partial^{\alpha} u\|_{M(W^{2,m}(\gamma)\rightarrow W^{2,m-|\alpha|}(\gamma))}\\
\leq&\sum_{\{\beta: \alpha \geq \beta \geq 0\}}\|u \|_{MW^{2,m-|\alpha|+|\beta|}(\gamma)}\\
\leq&\sum_{\{\beta: \alpha \geq \beta \geq 0\}}\|u \|^{\frac{m-|\alpha|+|\beta|}{m}}_{MW^{2,m}(\gamma)}\|u \|^{\frac{|\alpha|-|\beta|}{m}}_{ML^{2}(\gamma)}\\
\leq&\sum_{\{\beta: \alpha > \beta \geq 0\}}\|u \|^{\frac{m-|\alpha|+|\beta|}{m}}_{MW^{2,m}(\gamma)}\|u \|^{\frac{|\alpha|-|\beta|}{m}}_{ML^{2}(\gamma)}+\|u \|_{MW^{2,m}(\gamma)}.\\
\end{split}
\end{equation}
For any $\epsilon>0$, by Young's inequality, we have
\begin{align*}
&\sum_{\{\beta: \alpha > \beta \geq 0\}}\|u \|^{\frac{m-|\alpha|+|\beta|}{m}}_{MW^{2,m}(\gamma)}\|u \|^{\frac{|\alpha|-|\beta|}{m}}_{ML^{2}(\gamma)}\\
=&\sum_{\{\beta: \alpha > \beta \geq 0\}}{\epsilon}^{\frac{|\beta|-|\alpha|}{m}}
\|u \|^{\frac{m-|\alpha|+|\beta|}{m}}_{MW^{2,m}(\gamma)}(\epsilon\|u \|)^{\frac{|\alpha|-|\beta|}{m}}_{ML^{2}(\gamma)}\\
\lesssim&\sum_{\{\beta: \alpha > \beta \geq 0\}}\Big[\frac{m{\epsilon}^{\frac{m-|\beta|+|\alpha|}{m}}}{|\beta|-|\alpha|}
\|u \|_{MW^{2,m}(\gamma)}+\frac{m\epsilon}{|\beta|-|\alpha|}\|u \|_{ML^{2}(\gamma)}\Big]
\end{align*}
to complete the proof.
\end{proof}
\begin{lemma}\label{Lemma 4.4}
For any non-negative integer $m$ and $g\in L^2(\gamma)$, there is a set of functions $\{g_{\alpha}: |\alpha|\leq m\}$ such that
$$g=\sum_{|\alpha|\leq m}\partial^{\alpha}g_{\alpha}\quad\text{ and }\quad\|g_{\alpha}\|_{W^{2,m}(\gamma)}\lesssim \|g\|_{L^2(\gamma)}.$$
\end{lemma}
\begin{proof}
If $m=0$, then the conclusion is true. Suppose that the conclusion is true for $m=k$, we will prove that the conclusion is true for $m=k+1$.
For any $g\in L^2(\gamma),$ we know that
$$
g=\sum_{|\beta|\leq k}\partial^{\beta}g_{\beta}
$$
where $g\in W^{2,k}(\gamma)$ and $\|g_{\beta}\|_{W^{2,k}(\gamma)}\lesssim \|g\|_{L^2(\gamma)}$. Then
$g_{\beta}=(I-L)(I-L)^{-1}g_{\beta}.$
Since
$$I-L=\sum_{j=1}^{n}\partial_{x_j}(M_{x_j}-\partial_{x_j})-(n-1)I,$$
we have
$$
g_{\beta}=\sum_{j=1}^{n}\partial_{x_j}(M_{x_j}-\partial_{x_j})(I-L)^{-1}g_{\beta}-(n-1)(I-L)^{-1}g_{\beta}.
$$
By Theorem \ref{Theorem 3.1}, we know that $(I-L)^{-1}$ is bounded from $W^{2,k}(\gamma)$ to $W^{2,k+2}(\gamma)$, then $(I-L)^{-1}g_{\beta}\in W^{2,k+2}(\gamma)$. By \eqref{2.2}, we know that $(M_{x_j}-\partial_{x_j})$
is bounded from $W^{2,k+2}(\gamma)$ to $W^{2,k+1}(\gamma)$. We then obtain
$$
g=\sum_{|\beta|\leq k}\partial^{\beta}\left[\sum_{j=1}^{n}\partial_{x_j}(M_{x_j}-\partial_{x_j})(I-L)^{-1}g_{\beta} -(n-1)(I-L)^{-1}g_{\beta}\right],
$$
where
$$
\|(M_{x_j}-\partial_{x_j})(I-L)^{-1}g_{\beta}\|_{W^{2,k+1}(\gamma)}\lesssim\|(I-L)^{-1}g_{\beta}\|_{W^{2,k+2}(\gamma)}
\lesssim\|g_{\beta}\|_{W^{2,k}(\gamma)}\lesssim \|g\|_{L^2(\gamma)},
$$
and
$$
\|(n-1)(I-L)^{-1}g_{\beta}\|_{W^{2,k+1}(\gamma)}\lesssim\|(I-L)^{-1}g_{\beta}\|_{W^{2,k+2}(\gamma)}\lesssim \|g\|_{L^2(\gamma)}.
$$
We have completed the proof.
\end{proof}

For any $b\in{\mathbb{R}^n}$, let $W_b$ be an operator on $F^2$ such that
$$W_bh(z)=h(z-b)e^{z\cdot b-\frac{b^2}{2}},$$
for any $h\in F^2$.  This operator is the analogue of translation in the Fock space setting.

\begin{lemma}\label{lemma 3.1}
For any $b\in{\mathbb{R}^n}$, $W_b$ is a bounded operator on $F^{2,m}$ and
$$
\|W_{b}\|_{F^{2,m}}\leq c_{m,n}\left(\sum_{j=0}^m|b|^{2j}\right),
$$
where $c_{m,n}$ is a constant depend only on $m$ and $n$.
\end{lemma}
\begin{proof}
For any $h\in F^{2,m}$, we have
\begin{align*}
\|W_bh\|_{F^{2,m}}&\lesssim \left\||z|^m W_bh\right\|_{F^{2}}\\
&=\left[\int_{\mathbb{C}^n}|z|^{2m}|h(z-b)|^2e^{2z\cdot b-b^2} d\lambda(z)\right]^{1/2}\\
&=\left[\int_{\mathbb{C}^n}|z+b|^{2m}|h(z)|^2d\lambda(z)\right]^{1/2}\\
&\leq\left[\int_{\mathbb{C}^n}2^{m}(|z|^2+|b|^2)^{m}|h(z)|^2d\lambda(z)\right]^{1/2}\\
&\lesssim \left(\sum_{j=0}^m|b|^{2j}\right)\max_{0\leq k\leq m}\{ \||z|^kh\|_{F^2}\}\\
&\lesssim \left(\sum_{j=0}^m|b|^{2j}\right) \|h\|_{F^{2,m}},
\end{align*}
where the last inequality is due to Theorem \ref{theorem 2.3}.
\end{proof}

\begin{lemma}\label{Lemma 4.5}
Suppose that $u\in MW^{2,m}(\gamma)$ for some $m\geq0$, let
$$u_{r}(x)=\int_{\mathbb{R}^n}r^{-n}K(r^{-1}t)u(x-t)dt,$$
where $K\in C^{\infty}_{c}(\mathbb{B}^n)$, $K\geq0$ and $0\leq r\leq1$. Then
$$
 \sup_{0<r\leq1}\|u_{r}\|_{MW^{2,m}(\gamma)}\leq c_{m,n}\|u\|_{MW^{2,m}(\gamma)}
$$
and
$$
 \sup_{0<r\leq1}\|\partial^{\alpha}u_{r}\|_{M(W^{2,m}(\gamma)\rightarrow L^{2}(\gamma))}\leq
 c'_{m,n}\|\partial^{\alpha}u\|_{M(W^{2,m}(\gamma)\rightarrow L^{2}(\gamma))}
$$
for any $\alpha$ with $|\alpha|\leq m$,
where $c_{m,n}$ and $c'_{m,n}$ are constants that depend only on $m$ and $n$.
\end{lemma}
\begin{proof}For any  $g\in W^{2,m}(\gamma)$, by Minkowski's inequality, we have
\begin{align*}
\|u_{r}g\|_{W^{2,m}(\gamma)}&=\sum_{|\alpha|\leq m}\|\partial^{\alpha}(u_{r}g)\|_{L^{2}(\gamma)}\\
&=\sum_{|\alpha|\leq m}\left[\int_{\mathbb{R}^n}\left| \int_{\mathbb{R}^n}r^{-n}K(r^{-1}t)\partial^{\alpha} \big(u(x-t)g(x)\big)dt\right|^2d\gamma(x)\right]^{\frac{1}{2}}\\
&\leq\sum_{|\alpha|\leq m}\int_{\mathbb{R}^n} r^{-n}K(r^{-1}t)\left[\int_{\mathbb{R}^n}|\partial^{\alpha} \big(u(x-t)g(x)\big)|^2d\gamma(x)\right]^{\frac{1}{2}}dt.
\end{align*}
Let $\tau_{t}$ be the translation operator such that $\tau_{t}u(x)=u(x-t)$ and $M_{\tau_{t}u}$ be the multiplication operator, then
\begin{align*}
\|u_{r}g\|_{W^{2,m}(\gamma)}&\leq\sum_{|\alpha|\leq m}\int_{\mathbb{R}^n} r^{-n}K(r^{-1}t)\|M_{\tau_{t}u}g\|_{W^{2,m}(\gamma)}dt\\
&\leq c_m \|g\|_{W^{2,m}(\gamma)} \int_{|t|\leq r} r^{-n}K(r^{-1}t)\|M_{\tau_{t}u}\|_{MW^{2,m}(\gamma)}dt.
\end{align*}
We claim that $M_{\tau_{t}u}=G^{-1}W_{\frac{t}{2}}G M_{u}G^{-1}W_{\frac{-t}{2}}G$, then
$$
\|M_{\tau_{t}u}\|_{MW^{2,m}(\gamma)}\leq\|W_{\frac{t}{2}}\|_{F^{2,m}} \|u\|_{MW^{2,m}(\gamma)}\|W_{\frac{-t}{2}}\|_{F^{2,m}}.
$$
By Lemma \ref{lemma 3.1}, we have
\begin{align*}
\sup_{0<r\leq1}\|u_{r}\|_{MW^{2,m}(\gamma)}&\leq \sup_{0<r\leq1}c_m\int_{|t|\leq r} r^{-n}K(r^{-1}t) \|W_{\frac{t}{2}}\|_{F^{2,m}} \|W_{\frac{-t}{2}}\|_{F^{2,m}}  dt \|u\|_{MW^{2,m}(\gamma)}\\
&\leq c_{m,n} \|u\|_{MW^{2,m}(\gamma)}
\end{align*}
for some constant $c_{m,n}$.
Next, we prove the claim $M_{\tau_{t}u}=G^{-1}W_{\frac{t}{2}}G M_{u}G^{-1}W_{\frac{-t}{2}}G$. First, we show that $G^{-1}W_{\frac{t}{2}}G=M_{\exp[x\cdot \frac{t}{2}-\frac{t^2}{4}]}\tau_{t}$. For any $g\in W^{2,m}(\gamma)$,
we have
\begin{align*}
(W_{\frac{t}{2}}Gg)(z)
&=e^{z\cdot \frac{t}{2}-\frac{t^2}{8}}\int_{\mathbb{R}^{n}}g(x) e^{ x \cdot (z-\frac{t}{2})-\frac{(z-\frac{t}{2})^{2}}{2}}\frac{1}{{(2\pi)}^{\frac{n}{2}}}e^{-\frac{|x|^2}{2}}dx\\
&=e^{z\cdot \frac{t}{2}-\frac{t^2}{8}}\int_{\mathbb{R}^{n}}g(x-t) e^{ (x-t) \cdot (z-\frac{t}{2})-\frac{(z-\frac{t}{2})^{2}}{2}}\frac{1}{{(2\pi)}^{\frac{n}{2}}}e^{-\frac{|x-t|^2}{2}}dx\\
&=e^{-\frac{t^2}{4}}\int_{\mathbb{R}^{n}}g(x-t)e^{x\cdot \frac{t}{2}} e^{ x \cdot z-\frac{z^{2}}{2}}d\gamma\\
&=e^{-\frac{t^2}{4}}G[g(x-t)e^{x\cdot \frac{t}{2}}](z).
\end{align*}
Thus, we have
$$
(G^{-1}W_{\frac{t}{2}}Gg)(x)=e^{x\cdot \frac{t}{2}-\frac{t^2}{4}}g(x-t).
$$
Direct computation shows that $$M_{\tau_{t}u}=G^{-1}W_{\frac{t}{2}}G M_{u}G^{-1}W_{\frac{-t}{2}}G,$$
which completes the proof of the claim.

Similarly, for any $\alpha$ with $|\alpha|\leq m$ and $g\in L^2(\gamma)$, we have
\begin{align*}
\|(\partial^{\alpha}u_{r})g\|_{L^{2}(\gamma)}&\leq \int_{\mathbb{R}^n} r^{-n}K(r^{-1}t)\|M_{\tau_{t}\partial^{\alpha}u}g\|_{L^{2}(\gamma)}dt\\
&\leq c_m \|g\|_{W^{2,m}(\gamma)}\int_{|t|\leq r} r^{-n}K(r^{-1}t)\|M_{\tau_{t}\partial^{\alpha}u}\|_{M\big(W^{2,m}(\gamma)\rightarrow L^2(\gamma)\big)}dt.
\end{align*}
By the argument above, for any $\alpha$ with $|\alpha|\leq m$, we have
\begin{align*}
&\|M_{\tau_{t}\partial^{\alpha}u}\|_{M\big(W^{2,m}(\gamma)\rightarrow L^2(\gamma)\big)}\\
=&\|G^{-1}W_{\frac{t}{2}}G M_{\partial^{\alpha}u}G^{-1}W_{\frac{-t}{2}}G\|_{M\big(W^{2,m}(\gamma)\rightarrow L^2(\gamma)\big)}\\
\leq&\|G^{-1}W_{\frac{t}{2}}G\|_{M\big(L^{2}(\gamma)\rightarrow L^2(\gamma)\big)} \| M_{\partial^{\alpha} u}\|_{M\big(W^{2,m}(\gamma)\rightarrow L^2(\gamma)\big)} \|G^{-1}W_{\frac{-t}{2}}G\|_{M\big(W^{2,m}(\gamma)\rightarrow W^{2,m}(\gamma)\big)}\\
\leq & c'_{m,n}\| {\partial^{\alpha} u}\|_{M\big(W^{2,m}(\gamma)\rightarrow L^2(\gamma)\big)}
\end{align*}
for some constant $c'_{m,n}$, which completes the proof.
\end{proof}
\begin{proposition}\label{proposition 4.6}
If $u \in MW^{2,m}(\gamma),$ then $\partial^{\alpha}u \in M\big(W^{2,|\alpha|}(\gamma)\rightarrow L^2(\gamma)\big)$ for any $|\alpha|= m$ and
$u\in ML^2(\gamma)$. Moreover, we have
$$
\sum_{|\alpha|= m}\|\partial^{\alpha} u\|_{M(W^{2,|\alpha|}(\gamma) \rightarrow L^{2}(\gamma))}+\| u\|_{M(L^{2}(\gamma) )} \lesssim\|u\|_{MW^{2,m}(\gamma)}.
$$
\end{proposition}
\begin{proof}
First, we suppose that $u \in ML^{2}(\gamma)$.
For any $g\in W^{2,m}(\gamma)$ and multi-index $\alpha$ with $|\alpha|= m$, we have
\begin{align*}
&\|(\partial^{\alpha}u) g\|_{L^{2}(\gamma)}\\
=&\|\partial^{\alpha}(u g)-\sum_{\beta:0\leq\beta<\alpha}\partial^{\beta}u\partial^{\alpha-\beta}g\|_{L^{2}(\gamma)}\\
\leq&\|u g\|_{W^{2,|\alpha|}(\gamma)}+\|\sum_{\beta:0\leq\beta<\alpha}\partial^{\beta}u\partial^{\alpha-\beta}g\|_{L^{2}(\gamma)}\\
\leq& \|u\|_{MW^{2,|\alpha|}(\gamma)}\|g\|_{W^{2,|\alpha|}(\gamma)}
+\sum_{\beta:0\leq\beta<\alpha}\|\partial^{\beta}u\partial^{\alpha-\beta}g\|_{L^{2}(\gamma)}\\
\leq& \|u\|_{MW^{2,|\alpha|}(\gamma)}\|g\|_{W^{2,|\alpha|}(\gamma)}
+\sum_{\beta:0\leq\beta<\alpha}\|\partial^{\beta}u\|_{M(W^{2,|\beta|}(\gamma)\rightarrow L^{2}(\gamma))}\|\partial^{\alpha-\beta}g\|_{W^{2,|\beta|}(\gamma)}\\
\leq& \left[\|u\|_{MW^{2,|\alpha|}(\gamma)}
+\sum_{\beta:0\leq\beta<\alpha}\|\partial^{\beta}u\|_{M(W^{2,|\beta|}(\gamma)\rightarrow L^{2}(\gamma))}\right]\|g\|_{W^{2,m}(\gamma)}.\\
\end{align*}
By Lemma \ref{Lemma 4.3}, for any $\epsilon>0$ there is $c(\epsilon)$ such that
$$
\|\partial^{\beta} u\|_{M(W^{2,|\beta|}(\gamma) \rightarrow L^{2}(\gamma))}
\leq \varepsilon\|u\|_{ML^{2}(\gamma)}+c(\varepsilon)\|u\|_{MW^{2,|\beta|}(\gamma)}.
$$
Further, by Theorem \ref{Theorem 4.2}, we have
$$\|u\|_{MW^{2,|\alpha|}(\gamma)}\lesssim \|u\|_{M L^{2}(\gamma)} + \|u\|_{MW^{2,m}(\gamma)}.$$
Thus, we obtain
$$
\sum_{|\alpha|= m}\|\partial^{\alpha} u\|_{M(W^{2,m}(\gamma) \rightarrow L^{2}(\gamma))}\lesssim
\|u\|_{M L^{2}(\gamma)} + \|u\|_{MW^{2,m}(\gamma)}.
$$
Next, we will prove that $\|u\|_{M L^{2}(\gamma)}\lesssim \|u\|_{MW^{2,m}(\gamma)}$, which implies the conclusion.

For any $g\in L^2(\gamma)$, we have the decomposition $g=\sum_{|\alpha|\leq m}\partial^{\alpha}g_{\alpha}$ in Lemma \ref{Lemma 4.4}. Then
\begin{align*}
\|ug\|_{L^2(\gamma)}&\leq \sum_{|\alpha|\leq m}\|u\partial^{\alpha}g_{\alpha}\|_{L^2(\gamma)}\\
&=\sum_{|\alpha|\leq m}\Big\|\sum_{\{\beta: \alpha \geq \beta \geq 0\}} \frac{\alpha !}{\beta !(\alpha-\beta) !}\partial^{\beta}(g_{\alpha}(-\partial)^{\alpha-\beta} u)\Big\|_{L^2(\gamma)}\\
&\lesssim \sum_{|\alpha|\leq m}\sum_{\{\beta: \alpha \geq \beta \geq 0\}}\big\|\ \partial^{\beta}(g_{\alpha}(-\partial)^{\alpha-\beta} u)\big\|_{L^2(\gamma)}\\
&\lesssim \sum_{|\alpha|\leq m}\sum_{\{\beta: \alpha \geq \beta \geq 0\}}\big\|\ g_{\alpha}(-\partial)^{\alpha-\beta} u\big\|_{W^{2,m-|\alpha|+|\beta|}(\gamma)}\\
&\lesssim \sum_{|\alpha|\leq m}\sum_{\{\beta: \alpha \geq \beta \geq 0\}}\big\|\partial^{\alpha-\beta} u\big\|_{M\big(W^{2,m}(\gamma)\rightarrow W^{2,m-|\alpha|+|\beta|}(\gamma)\big)}\| g_{\alpha}\|_{W^{2,m}(\gamma)}\\
&\lesssim \sum_{|\alpha|\leq m}\sum_{\{\beta: \alpha \geq \beta \geq 0\}}\big\|\partial^{\alpha-\beta} u\big\|_{M\big(W^{2,m}(\gamma)\rightarrow W^{2,m-|\alpha|+|\beta|}(\gamma)\big)}\| g\|_{L^{2}(\gamma)}.
\end{align*}
By Lemma \ref{Lemma 4.3} and the inequality above, for any $0<\epsilon<1$, there is $c(\epsilon)$ such that
$$\|u\|_{ML^{2}(\gamma)}
\lesssim \varepsilon\|u\|_{ML^{2}(\gamma)}+c(\varepsilon)\|u\|_{MW^{2,m}(\gamma)}.$$
Then, we have $\|u\|_{ML^{2}(\gamma)} \lesssim\|u\|_{MW^{2,m}(\gamma)}.$

Next, we remove the hypothesis. For any $r>0$, let $u_r$ be the function in Lemma \ref{Lemma 4.5}. Thus $u_r$ is in $C^{\infty}(\mathbb{R}^n)$.
We can choose a set of smooth function $\phi_r$ such that $\phi_r(x)=1$ when $|x|\leq\frac{1}{r}$, $\phi_r(x)=0$ when $|x|>\frac{1}{r}+1$ and
$$\sum_{|\alpha|\leq m}\sup_{x}|\partial^{\alpha}\phi_r(x)|\leq c,$$
where $c$ is independent with $r$. We know that $\phi_ru_{r}$ is bounded, thus $\phi_ru_{r}\in ML^2(\gamma)$. By the conclusion above we know that
$$\|\phi_ru_{r}\|_{ML^2(\gamma)}\leq c' \|\phi_ru_{r}\|_{MW^{2,m}(\gamma)},$$
where $c'$ is an absolute constant.
Since $\lim_{r\rightarrow0}\phi_ru_{r}=u$ almost everywhere.
Thus for any $g\in L^{2}(\gamma),$ we have
$$
\|ug\|_{L^2(\gamma)}\leq\liminf_{r\rightarrow0}\|\phi_ru_{r}g\|_{L^2(\gamma)}.
$$
Then by Lemma \ref{Lemma 4.5} and Lemma \ref{lemma 3.3}, we have
\begin{align*}
\|u\|_{L^2(\gamma)}&\leq\liminf_{r\rightarrow0}\|\phi_ru_{r}\|_{L^2(\gamma)}\lesssim \liminf_{r\rightarrow0}\|\phi_ru_{r}\|_{MW^{2,m}(\gamma)}\\
&\leq \liminf_{r\rightarrow0}\|\phi_r\|_{MW^{2,m}(\gamma)}\|u_{r}\|_{MW^{2,m}(\gamma)}\leq c_{m,n}c \|u\|_{MW^{2,m}(\gamma)}
\end{align*}
to complete the proof of the claim.
\end{proof}

To prove our main theorem in the next section, we need the following theorem about multipliers in the Gauss-Sobolev space.
\begin{theorem}\label{theorem 3.6}
$u\in MW^{2,m}(\gamma)$ if and only if $\partial^{\alpha}u \in M\big(W^{2,|\alpha|}(\gamma)\rightarrow L^2(\gamma)\big)$ for any $|\alpha|= m$ and
$u\in ML^2(\gamma)$. In this case, we have
$$
\|u\|_{MW^{2,m}(\gamma)}\simeq \sum_{|\alpha|=m}\|\partial^{\alpha} u\|_{M(W^{2,|\alpha|}(\gamma) \rightarrow L^{2}(\gamma))}+\|u\|_{ML^2(\gamma)}.
$$
\end{theorem}

\begin{proof}[Proof of Theorem \ref{theorem 3.6}]
If $\partial^{\alpha}u \in M\big(W^{2,|\alpha|}(\gamma)\rightarrow L^2(\gamma)\big)$ for any $|\alpha|= m$ and
$u\in ML^2(\gamma)$.
Let $u_r$ be the function corresponding to $u$ as in Lemma \ref{Lemma 4.5}. Since $u\in ML^2(\gamma)$, we know that $u$ is bounded.
It is easy to prove that
$$\sum_{|\alpha|\leq m}\sup_{x}|\partial^{\alpha}u_r(x)|<\infty$$
for any $r>0$, thus $\|u_r\|_{MW^{2,m}(\gamma)}<\infty.$ Then for any $g\in W^{2,m}(\gamma)$, we have
\begin{align*}
&\|u_{r}g\|_{W^{2,m}(\gamma)}\\
=&\sum_{|\alpha|\leq m}\|\partial^{\alpha}(u_{r}g)\|_{L^2(\gamma)}\\
\leq&\sum_{|\alpha|\leq m}\sum_{0\leq\beta\leq\alpha}\|\partial^{\beta}u_{r} \partial^{\alpha-\beta}g\|_{L^2(\gamma)}\\
=&\sum_{|\alpha|\leq m}\sum_{0\leq\beta\leq\alpha}\|\partial^{\beta}u_{r}\|_{M\big( W^{2,|\beta|}(\gamma)\rightarrow L^2(\gamma)\big)} \|\partial^{\alpha-\beta}g\|_{W^{2,|\beta|}(\gamma)}\\
\lesssim&\left[\sum_{0\leq|\beta|< m}\|\partial^{\beta}u_{r}\|_{M\big( W^{2,|\beta|}(\gamma)\rightarrow L^2(\gamma)\big)}+\sum_{|\beta|= m}\|\partial^{\beta}u_{r}\|_{M\big( W^{2,|\beta|}(\gamma)\rightarrow L^2(\gamma)\big)} \right]\|g\|_{W^{2,m}(\gamma)}.
\end{align*}
By Lemma \ref{Lemma 4.3} and Theorem \ref{Theorem 4.2}, for any $\epsilon>0$, there is a $c(\epsilon)$ such that
\begin{align*}
&\sum_{0\leq|\beta|< m}\|\partial^{\beta} u_r\|_{M(W^{2,|\beta|}(\gamma)\rightarrow L^{2}(\gamma))}\\
\lesssim &\sum_{\{\beta: 0\leq|\beta|<m\}}\|u_r \|_{MW^{2,|\beta|}(\gamma)}\\
\lesssim &\epsilon\|u_r\|_{MW^{2,m}(\gamma)}+c(\epsilon)\|u_r\|_{ML^{2}(\gamma)}.
\end{align*}
Then we obtain
$$
\|u_{r}\|_{MW^{2,m}(\gamma)}\lesssim\epsilon\|u_r\|_{MW^{2,m}(\gamma)}+c(\epsilon)\|u_r\|_{ML^{2}(\gamma)}+\sum_{|\beta|= m}\|\partial^{\beta}u_{r}\|_{M\big( W^{2,|\beta|}(\gamma)\rightarrow L^2(\gamma)\big)}.
$$
Let $\epsilon$ be small enough, then we get
$$
\|u_{r}\|_{MW^{2,m}(\gamma)}\lesssim\|u_r\|_{ML^{2}(\gamma)}+\sum_{|\beta|= m}\|\partial^{\beta}u_{r}\|_{M\big( W^{2,|\beta|}(\gamma)\rightarrow L^2(\gamma)\big)}.
$$
By Lemma \ref{Lemma 4.5}, we have
\begin{align*}
\|u\|_{MW^{2,m}(\gamma)}&\leq \liminf_{r\rightarrow 0}\|u_r\|_{MW^{2,m}(\gamma)} \\
&\lesssim  \liminf_{r\rightarrow 0}\|u_r\|_{ML^{2}(\gamma)}+\liminf_{r\rightarrow 0}\sum_{|\beta|= m}\|\partial^{\beta}u_{r}\|_{M\big( W^{2,|\beta|}(\gamma)\rightarrow L^2(\gamma)\big)}\\
&\lesssim  \|u\|_{ML^{2}(\gamma)}+\sum_{|\beta|= m}\|\partial^{\beta}u\|_{M\big( W^{2,|\beta|}(\gamma)\rightarrow L^2(\gamma)\big)}.
\end{align*}
The converse is due to Proposition \ref{proposition 4.6}.
\end{proof}

\section{Applications to Certain Operators on the Fock-Sobolev Space}

In this section, we study the boundedness of $S_{\varphi}$. We need several lemmas.
Let $C_i$ and $C_{-i}$ be composition operators on $F^2$ such that for any $f\in F^2$
$$
C_{i}f(z)=f(iz)\quad\text{and}\quad C_{-i}f(z)=f(-iz).
$$
It is easy to show that $C_i$ and $C_{-i}$ are isometries on $F^{2,m}$ for any $m\in\mathbb{N}$.
\begin{lemma}\label{lemma 3.2}
For any $a\in\mathbb{R}^n$, let $M_{e^{ia\cdot x}}$ be the multiplication operator on $W^{2,m}(\gamma)$. If $S_{\varphi}$ is bounded on $F^{2,m}$, then $G^{-1}C_{-i}S_{\varphi}C_{i}G$ commutes with $M_{e^{-ia\cdot x}}$.
\end{lemma}
\begin{proof}
By \cite[Lemma 3.3]{wick2019}, we know that $S_{\varphi}$ commutes with $W_a$ on $F^2$. Since $W_a$ is bounded on $F^{2,m}$,
we know that $S_{\varphi}$ commutes with $W_a$ on $F^{2,m}$.
Then $G^{-1}C_{-i}S_{\varphi}C_{i}G$ commutes with $G^{-1}C_{-i} W_a C_{i}G$. We only need to show that
$$G^{-1}C_{-i} W_a C_{i}G=M_{e^{-ia\cdot x}}.$$
For any $f\in F^{2,m}$ and $z\in \mathbb{C}^n$, we have
$$
C_{-i} W_a C_{i}f(z)=f(z-ia)e^{-i z\cdot a -\frac{a^2}{2}}.
$$
On the other hand
\begin{align*}
M_{e^{-ia\cdot x}}G^{-1}f(x)&=e^{-ia\cdot x}\int_{\mathbb{C}^{n}}f(z) e^{ x \cdot \overline{z}-\frac{\overline{z}^{2}}{2}} d\lambda(z)\\
&=e^{-ia\cdot x}\int_{\mathbb{C}^{n}}f(z) e^{ x \cdot \overline{z}-\frac{\overline{z}^{2}}{2}} \pi^{-n} e^{-|z|^{2}} d v(z)\\
&=e^{-ia\cdot x}\int_{\mathbb{C}^{n}}f(z-ia) e^{ x \cdot \overline{(z-ia)}-\frac{\overline{(z-ia)}^{2}}{2}} \pi^{-n} e^{-|z-ia|^{2}} d v(z)\\
&=\int_{\mathbb{C}^{n}}f(z-ia)e^{-iz\cdot a-\frac{a^2}{2}} e^{ x \cdot \overline{z}-\frac{\overline{z}^{2}}{2}} d \lambda(z)\\
&=G^{-1}[f(z-ia)e^{-iz\cdot a-\frac{a^2}{2}}](x).
\end{align*}
Then
$$
GM_{e^{-ia\cdot x}}G^{-1}f(z)=g(z-ia)e^{-iz\cdot a-\frac{a^2}{2}}=C_{-i} W_a C_{i}f(z),
$$
which completes the proof.
\end{proof}

Let $C_{p}^{\infty}(\mathbb{R}^n)$ denote the set of smooth function $f$ such that there is a positive number $N=N_f$, such that
$$
f(x+2Ny)=f(x)
$$
for any $x\in[-N,N]^n$ and $y\in \mathbb{Z}^n$, moreover, $f(x)=0$ when $x\in [-N,N]^n \setminus \left[-\frac{N}{2\sqrt{n}},\frac{N}{2\sqrt{n}}\right]^n$.
We call $N_f$ the period of $f$.

\begin{lemma}\label{lemma 3.4}
For any $f\in C_{p}^{\infty}(\mathbb{R}^n)$, there is a sequence $f_n\in \operatorname{span}\{e^{ia\cdot x}: a\in \mathbb{R}^n\}$ such that
$$\lim_{n\rightarrow\infty}\|M_{f_n}-M_f\|_{MW^{2,m}(\gamma)}=0.$$
\end{lemma}
\begin{proof}By \cite[Theorem 2.11 and Corollary 1.9, Chapter 7]{stein1971}, there is a sequence of functions $\{f_n\}\subset \operatorname{span}\{e^{ia\cdot x}: a\in \mathbb{R}^n\} $
such that
$$
\lim_{n\rightarrow\infty}\sup_{x}|\partial^{\alpha}f(x)-\partial^{\alpha}f_n(x)|=0,
$$
for any $\alpha\in \mathbb{R}^n$ with $|\alpha|\leq m$. By Lemma \ref{lemma 3.3}, we obtain the conclusion.
\end{proof}

\begin{lemma}\label{lemma 3.5}
$C_{p}^{\infty}(\mathbb{R}^n)$ is a dense subset of $W^{2,m}(\gamma)$.
\end{lemma}
\begin{proof}First, we show that $C_{p}^{\infty}(\mathbb{R}^n)$ is contained in $W^{2,m}(\gamma)$.
For any $f\in C_{p}^{\infty}(\mathbb{R}^n)$ and any $\alpha\in \mathbb{N}^n$, let $N=N_f$ be the period of $f$, we have
\begin{align*}
&\int_{\mathbb{R}^n}|\partial^{\alpha}f(x)|^2d\gamma(x)\\
=&\sum_{y\in\mathbb{Z}^n}\int_{[-N,N]^n+2Ny}|\partial^{\alpha}f(x)|^2d\gamma(x)\\
=&\sum_{y\in\mathbb{Z}^n}\int_{[-N,N]^n}|\partial^{\alpha}f(x)|^2\frac{1}{{(2\pi)}^{\frac{n}{2}}}e^{-\frac{|x+2Ny|^2}{2}}dx\\
=&\sum_{y\in\mathbb{Z}^n\setminus\{0\}}\int_{[-N,N]^n}|\partial^{\alpha}f(x)|^2\frac{1}{{(2\pi)}^{\frac{n}{2}}}e^{-\frac{|x+2Ny|^2}{2}}dx
+\int_{[-N,N]^n}|\partial^{\alpha}f(x)|^2d\gamma(x)\\
=&\sum_{y\in\mathbb{Z}^n\setminus\{0\}}\int_{\left[-\frac{N}{2\sqrt{n}},\frac{N}{2\sqrt{n}}\right]^n}|\partial^{\alpha}f(x)|^2
\frac{1}{{(2\pi)}^{\frac{n}{2}}}e^{-\frac{|x+2Ny|^2}{2}}dx
+\int_{\left[-\frac{N}{2\sqrt{n}},\frac{N}{2\sqrt{n}}\right]^n}|\partial^{\alpha}f(x)|^2d\gamma(x).
\end{align*}
When $x\in\left[-\frac{N}{2\sqrt{n}},\frac{N}{2\sqrt{n}}\right]^n$ and $y\in\mathbb{Z}^n\setminus\{0\}$, we have
$$|x|\leq \frac{N}{2}\leq \frac{N|y|}{2}.$$
Then
$$e^{-\frac{|x+2Ny|^2}{2}}\leq e^{-\frac{|x|^2}{2}+2N|x| |y|-2N^2|y|^2}\leq e^{-\frac{|x|^2}{2}+2N\frac{N|y|}{2} |y|-2N^2|y|^2} \leq e^{-\frac{|x|^2}{2}-N^2|y|^2}.$$
That is to say
$$
\sum_{y\in\mathbb{Z}^n\setminus\{0\}}\int_{\left[-\frac{N}{2\sqrt{n}},\frac{N}{2\sqrt{n}}\right]^n}|\partial^{\alpha}f(x)|^2
\frac{1}{{(2\pi)}^{\frac{n}{2}}}e^{-\frac{|x+2Ny|^2}{2}}dx
\leq\sum_{y\in\mathbb{Z}^n\setminus\{0\}}e^{-N^2|y|^2}\int_{\left[-\frac{N}{2\sqrt{n}},\frac{N}{2\sqrt{n}}\right]^n}|\partial^{\alpha}f(x)|^2d\gamma(x).
$$
Since
\begin{align*}
\sum_{y\in\mathbb{Z}^n\setminus\{0\}}e^{-N^2|y|^2}&\leq\sum_{j=0}^{n}\sum_{y\in\mathbb{Z}^n,y_j\neq0}e^{-N^2|y|^2}\\
&=n\sum_{y_1=1}^{\infty}\sum_{y_2=0}^{\infty}\dots\sum_{y_n=0}^{\infty}e^{-N^2|y|^2}\\
&=n(\sum_{y_1=1}^{\infty}e^{-N^2|y_1|^2})(\sum_{y_2=0}^{\infty}e^{-N^2|y_2|^2})\dots(\sum_{y_n=0}^{\infty}e^{-N^2|y_n|^2})\\
&\leq\frac{n e^{-N^2}}{(1-e^{-N^2})^n},
\end{align*}
which implies that $$\int_{\mathbb{R}^n}|\partial^{\alpha}f(x)|^2d\gamma(x)\leq
\left(\frac{n e^{-N^2}}{(1-e^{-N^2})^n}+1\right)\int_{\left[-\frac{N}{2\sqrt{n}},\frac{N}{2\sqrt{n}}\right]^n}|\partial^{\alpha}f(x)|^2d\gamma(x)<\infty.$$

On the other hand, since $C_0^{\infty}(\mathbb{R}^n)$ is dense in $W^{2,m}(\gamma)$, we only need
approximate any $g\in C_0^{\infty}(\mathbb{R}^n)$. For any $\epsilon>0$, there is an positive integer $N$ such that
$$
g(x)=0, \text{ when } x\in \mathbb{R}^n\setminus \left[-\frac{N}{2\sqrt{n}},\frac{N}{2\sqrt{n}}\right]^n
$$
and
$$\sum_{y\in\mathbb{Z}^n\setminus\{0\}}e^{-N^2|y|^2}\int_{\mathbb{R}^n}|\partial^{\alpha}g(x)|^2d\gamma(x)\leq \epsilon^2,$$
for any $\alpha\in \mathbb{N}^n$ with $|\alpha|\leq m$.
Let
$$f(x)=\sum_{y\in \mathbb{Z}^n}g(x+2Ny).$$
Then, we know that $f\in C_{p}^{\infty}(\mathbb{R}^n)$ and
$$f(x)=g(x),\text{ when }x\in  \left[-\frac{N}{2\sqrt{n}},\frac{N}{2\sqrt{n}}\right]^n.$$
Then
\begin{align*}
\|g-f\|_{W^{2,m}(\gamma)}
&=\sum_{|\alpha|\leq m}\left[\int_{\mathbb{R}^n}|\partial^{\alpha}g(x)-\partial^{\alpha}f(x)|^2d\gamma(x)\right]^{1/2}
=\sum_{|\alpha|\leq m}\left[\int_{\mathbb{R}^n\setminus\left[
-N,N\right]^n}|\partial^{\alpha}f(x)|^2d\gamma(x)\right]^{1/2}.\\
\end{align*}
By the argument above, we know that
\begin{align*}
&\int_{\mathbb{R}^n\setminus[-N,N]^n}|\partial^{\alpha}f(x)|^2d\gamma(x)\\
\leq&\sum_{y\in\mathbb{Z}^n\setminus\{0\}}e^{-N^2|y|^2}\int_{\left[-\frac{N}{2\sqrt{n}},\frac{N}{2\sqrt{n}}\right]^n}|\partial^{\alpha}f(x)|^2d\gamma(x)\\
=&\sum_{y\in\mathbb{Z}^n\setminus\{0\}}e^{-N^2|y|^2}
\int_{[-\frac{N}{2\sqrt{n}},\frac{N}{2\sqrt{n}}]^n}|\partial^{\alpha}g(x)|^2d\gamma(x)\\
\leq&\sum_{y\in\mathbb{Z}^n\setminus\{0\}}e^{-N^2|y|^2}
\left[\int_{\mathbb{R}^n}|\partial^{\alpha}g(x)|^2d\gamma(x)\right]\\
\leq&\epsilon^2.
\end{align*}
Then, we have
$$\|g-f\|_{W^{2,m}(\gamma)}\leq c_{m}\epsilon,$$
where $c_{m}=\operatorname{card}\{\alpha: |\alpha|\leq m\}$. We have completed the proof.
\end{proof}

We can now give a characterization of the boundedness of $S_{\varphi}$ on $F^{2,m}$.  This is the analogue of the result in \cite{wick2019} obtained for the Fock space $F^2$.

\begin{theorem}\label{theorem 3.7}
Let $m$ be a positive integer, $S_{\varphi}$ is bounded on $F^{2,m}$ if and only if
$$
S_{\varphi}=C_{i}GM_uG^{-1}C_{-i},
$$
where $u$ is a multiplier on $W^{2,m}(\gamma)$. In this case,
$$
\varphi(z)=\int_{\mathbb{R}^{n}} u(2x) e^{-2\left(x-\frac{i}{2} z\right) \cdot\left(x-\frac{i}{2} z\right)} d x.
$$
\end{theorem}
\begin{proof}Recall that $G$ and $C_{i}$ are isometries.
If $S_{\varphi}=C_{i}GM_uG^{-1}C_{-i}$, where $u$ is a multiplier, then $S_{\varphi}$ is bounded.

On the other hand, suppose that $S_{\varphi}$ is bounded. By Lemma \ref{lemma 3.2} and Lemma \ref{lemma 3.4},
we know that for any $h\in C_{p}^{\infty}(\mathbb{R}^n)$,  $G^{-1}C_{-i}S_{\varphi}C_{i}G$ commutes with $M_h$. Let
$$u=G^{-1}C_{-i}S_{\varphi}C_{i}G1.$$
Then
$$G^{-1}C_{-i}S_{\varphi}C_{i}Gh=G^{-1}C_{-i}S_{\varphi}C_{i}GM_h 1=M_h u=M_{u}h.$$
Since $C_{p}^{\infty}(\mathbb{R}^n)$ is a dense subset of $W^{2,m}(\gamma)$, by Lemma \ref{lemma 3.5}, we know that
$$G^{-1}C_{-i}S_{\varphi}C_{i}G=M_u.$$
That is to say
$S_{\varphi}=C_{i}GM_uG^{-1}C_{-i}$,
where $u$ is a multiplier on $W^{2,m}(\gamma)$.

Next, we prove the second part. By Theorem \ref{theorem 3.6}, $u$ is in $M{L^2(\gamma)}=L^\infty$,
that is to say that
$S_{\varphi}=C_{i}GM_uG^{-1}C_{-i}$ is bounded on $F^2$. Then, by \cite[Proposition 3.6 and Theorem 1.1]{wick2019}, we have
$$
S_{\varphi}=B\mathcal{F}^{-1}M_{v}\mathcal{F}B^{-1}\text{ and }\varphi(z)=\int_{\mathbb{R}^{n}} v(x) e^{-2\left(x-\frac{i}{2} z\right) \cdot\left(x-\frac{i}{2} z\right)} d x,
$$
where $\mathcal{F}$ is the Fourier transform and $M_{v}$ is a multiplication operator with $v\in L^{\infty}(\mathbb{R}^n)$.
On the Fock space, by \cite[Lemma 2.3]{wick2019}, we have $C_{i}= B\mathcal{F}^{-1}B^{-1}$. By Proposition \ref{proposition 2.2}, we have
\begin{align*}
S_{\varphi}&=C_{i}GM_uG^{-1}C_{-i}\\
&=B\mathcal{F}^{-1}B^{-1}GM_uG^{-1}B\mathcal{F}B^{-1}\\
&=B\mathcal{F}^{-1}C^{-1}_{\frac{1}{2}}M^{-1}G^{-1}GM_uG^{-1}GMC_{\frac{1}{2}}\mathcal{F}B^{-1}\\
&=B\mathcal{F}^{-1}C^{-1}_{\frac{1}{2}}M_uC_{\frac{1}{2}}\mathcal{F}B^{-1}\\
&=B\mathcal{F}^{-1}M_{C^{-1}_{\frac{1}{2}}u}\mathcal{F}B^{-1}.
\end{align*}
By the argument above we obtain $v(x)=C^{-1}_{\frac{1}{2}}u=u(2x)$.
\end{proof}

\subsection{Other Operator Theoretic Properties}
According to the theorems above, we can obtain some properties of $S_{\varphi}$ on the Fock-Sobolev space.
\begin{corollary}For any $m>0$, if $S_{\varphi}$ is bounded on $F^{2,m}$, we have following conclusions.
\begin{enumerate}
\item The set of operators $\{S_{\varphi}: S_{\varphi} \text{ is bounded}\}$ is a commutative algebra.
\item $S_{\varphi}$ is compact on $F^{2,m}$ if and only if $S_{\varphi}=0$.
\item $S_{\varphi}$ is invertible on $F^{2,m}$ if and only if $\frac{1}{u}$ is essentially bounded, where $u$ is the multiplier on $W_{2,m}(\gamma)$corresponding to $S_{\varphi}$ in Theorem \ref{theorem 3.7}.
\end{enumerate}
\end{corollary}
\begin{proof}(1) follows form Theorem \ref{theorem 3.7} and the fact that the set of multiplication operators is a commutative algebra.

To prove (2), we need a fact. For any smooth function $\eta$ with compact support, there is a sequence of functions $f_n$ such that
$$f_n\rightarrow0 \text{ weakly and }\|f_{n}\|_{W^{2,m}(\gamma)}=\|\eta\|_{L^2(\gamma)}+O(n^{-1}),$$
 moreover, if $u\in MW^{2,m}(\gamma)$, then
$$ \|uf_{n}\|_{W^{2,m}(\gamma)}=\|u\eta\|_{L^2(\gamma)}+O(n^{-1}).$$
For the construction see \cite[pg. 270]{Mazya}. Although the construction is made for the Sobolev space, the proof is also valid for Gauss-Sobolev space.
If $u \in MW^{2,m}(\gamma)$ is compact, then
$$\lim_{n\rightarrow \infty} \|uf_{n}\|_{W^{2,m}(\gamma)}=0.$$
That is to say $\|u\eta\|_{L^2(\gamma)}=0$, which implies that $u=0$. By Theorem \ref{theorem 3.7}, we get the conclusion.

Next we prove (3). If $\frac{1}{u}$ is essentially bounded, we claim that $\frac{1}{u}$ is also a multiplier on $W^{2,m}(\gamma)$.
For any $\alpha$ with $|\alpha|=m$, we have
$$
\partial^{\alpha}\frac{1}{u}=\sum_{\beta^1+\cdots +\beta^m\leq\alpha }c_{\beta^1,\cdots,\beta^m,\alpha}\frac{\partial^{\beta^1}u\cdots\partial^{\beta^{m}}u}{u^{m+1}},
$$
where $\{c_{\beta^1,\cdots,\beta^m,\alpha}\}$ are some constants.
By Lemma \ref{Lemma 4.3}, we have $\partial^{\beta^1}u\cdots\partial^{\beta^{m}}u$ is a multiplier from $W^{2,m}(\gamma)$ to $L^2(\gamma)$ for any $\beta^1,\cdots,\beta^m$ with $\beta^1+\cdots +\beta^m\leq\alpha$ , which implies that $\partial^{\alpha}\frac{1}{u}$ is a multiplier from $W^{2,m}(\gamma)$ to $L^2(\gamma)$. By Theorem \ref{theorem 3.6}, we obtain that $\frac{1}{u}$ is a multiplier on $W^{2,m}(\gamma)$. Then $M_{\frac{1}{u}}$ is the inverse operator of $M_{u}$, which implies that $S_{\varphi}$ is invertible.

On the other hand, if $S_{\varphi}$ is invertible on $F^{2,m}$, then $M_{u}$ is invertible on $W^{2,m}(\gamma)$. For any $g\in W^{2,m}(\gamma)$, there is $f\in W^{2,m}(\gamma)$ such that $g=uf$. Then $\frac{1}{u}g=f\in W^{2,m}(\gamma)$. Since $M_{\frac{1}{u}}$ is a closed operator, we have $M_{\frac{1}{u}}$ is bounded on $W^{2,m}(\gamma)$. By Theorem \ref{theorem 3.6}, we know that $M_{\frac{1}{u}}$ is bounded on $L^{2}(\gamma)$. That is to say $\frac{1}{u}$ is essentially bounded.
\end{proof}

\bibliographystyle{amsplain}

\end{document}